\documentclass[12pt]{amsart}

\usepackage{amsmath,amssymb}
\usepackage{amsthm}
\usepackage{hyperref}
\usepackage[margin=1.1in]{geometry}

\numberwithin{equation}{section}

\newtheorem{prop}{Proposition}
\newtheorem{lemma}[prop]{Lemma}

\newtheorem{thm}[prop]{Theorem}

\numberwithin{prop}{section}

\theoremstyle{definition}

\newcommand{\del}{\partial}
\newcommand{\delb}{\bar{\partial}}\newcommand{\dt}{\frac{\partial}{\partial t}}
\newcommand{\brs}[1]{\left| #1 \right|}

\newcommand{\gD}{\Delta}
\newcommand{\gd}{\delta}

\newcommand{\gU}{\Upsilon}

\newcommand{\gw}{\omega}
\newcommand{\ga}{\alpha}
\newcommand{\gb}{\beta}

\newcommand{\N}{\nabla}

\renewcommand{\bar}[1]{\overline{#1}}

\renewcommand{\i}{\sqrt{-1}}

\newcommand{\bga}{\bar{\ga}}

\newcommand{\bj}{\bar{j}}

\newcommand{\bl}{\bar{l}}

\newcommand{\bq}{\bar{q}}

\newcommand{\IP}[1]{\left<#1\right>}

\DeclareMathOperator{\tr}{tr}

\DeclareMathOperator{\Area}{Area}

\begin{document}

\title[Complex manifolds with negative curvature operator]{Complex manifolds with negative curvature operator}

\begin{abstract} We prove that compact complex manifolds with admitting metrics with negative Chern curvature operator either admit a $d d^c$-exact positive $(1,1)$ current, or are K\"ahler with ample canonical bundle.  In the case of complex surfaces we obtain a complete classification.  The proofs rely on a global existence and convergence result for the pluriclosed flow.
\end{abstract}

\date{\today}

\author{Man-Chun Lee}
\address{Mathematics Department\\
         University of British Columbia\\
         Vancouver, BC V6T 1Z2}
\email{\href{mailto:mclee@math.ubc.ca}{mclee@math.ubc.ca}}

\author{Jeffrey Streets}
\address{Rowland Hall\\
         University of California\\
         Irvine, CA 92617}
\email{\href{mailto:jstreets@uci.edu}{jstreets@uci.edu}}

\maketitle

\section{Introduction}

 Since the original work of Hamilton classifying three-manifolds with positive Ricci curvature \cite{Hamilton3folds}, the Ricci flow has found many applications in understanding the topology of manifolds satisfying certain curvature bounds.  Examples include the proof of the $\tfrac{1}{4}$-pinching differential sphere theorem \cite{BrendleSchoen}, the classification of manifolds with positive curvature operator \cite{Hamilton4folds, BohmWilking}, and others \cite{Margerin, NiWilking}.  In the setting of complex geometry, the K\"ahler-Ricci flow has been used to give a proof of the Frankel Conjecture \cite{ChenSunTian}, classifying K\"ahler manifolds with positive bisectional curvature operator (cf. also \cite{GuFrankel, MokFrankel}).  Some recent works have used the Hermitian Curvature Flow, an extension of the K\"ahler-Ricci flow to the setting of complex, non-K\"ahler geometry introduced by the second author and Tian \cite{HCF}, to prove new classification results for complex manifolds admitting Hermitian metrics satisfying certain curvature conditions.  For instance Ustinovskiy \cite{Ustinovskiy} gave an extension of the Frankel Conjecture, proving that complex manifolds admitting Hermitian metrics with weakly positive bisectional curvatures, positive at one point, are biholomorphic to projective space.  

While the results mentioned above focus on positivity conditions for curvature, various notions of negative curvature have appeared in complex geometry.  Inspired by ideas from Kobayashi hyperbolicity, Yau conjectured that if a projective manifold admits a K\"ahler metric with negative holomorphic sectional curvature then the canonical bundle of the manifold is ample. Very recently, Wu and Yau \cite{wuyau2016} had confirmed the conjecture. Shortly after, the assumption on projectivity was removed by Tosatti and Yang \cite{TosattiYang2017}. Later on, Diverio and Trapani \cite{DiverioTrapani2019} further weakened the assumptions to quasi-negative holomorphic sectional curvature. See \cite{WuYau2016-1} for a independent proof by Wu and Yau and also \cite{guenancia2018, HeierLuWong2016, HeierLuWongZheng2018,yang2019real} for related works.   Also the first author used Hermitian curvature flows to show that the canonical line bundle of a compact Hermitian manifold with nonpositive bisectional curvature and quasi-negative Chern-Ricci curvature is ample \cite{LeeNeg}.

Our main result classifies compact complex surfaces admitting a Hermitian metric with negative Chern curvature operator.  Specifically, given $(M^{2n}, J, h)$ a Hermitian manifold, let $\Omega^h$ denote the curvature tensor of the associated Chern connection.  We say that $h$ has non-positive (resp. negative) curvature operator if there exists $\gd \geq 0$ (resp $\gd > 0$) so that for any section $\xi \in \Lambda^{1,1} T$, one has
\begin{align} \label{f:curvcond}
\xi^{\bj i} \xi^{\bl k} \Omega^h_{i \bj k \bl} \leq - \gd \brs{\xi}_h^2.
\end{align}
In fact, our results will hold in slightly greater generality.  We only require that the inequality (\ref{f:curvcond}) hold on the cone of nonnegative $(1,1)$ tensors $\xi$.  This condition appeared in the recent work by Yang and Zheng \cite{yang2019real} where it is referred to as ``negative real bisectional curvature,'' and is quite natural from the point of view of the Schwarz Lemma.  When the metric $h$ is K\"ahler, this condition is equivalent to negative holomorphic sectional curvature.  In general it is slightly stronger than negative holomorphic sectional curvature (cf. \cite{yang2019real}), though it is significantly weaker than the condition of negative \emph{holomorphic bisectional curvature}, which automatically implies that the underlying complex manifold is K\"ahler.  Our classification follows from a general structural result yielding information on the underlying complex geometry obtained from the existence of a metric with nonpositive curvature operator:

\begin{thm} \label{t:mainthm} Let $(M^{2n}, J, h)$ be a compact Hermitian manifold with nonpositive curvature operator.  Then one of the following must hold:
\begin{enumerate}
\item $(M^{2n}, J)$ admits a nonzero positive $(1,1)$ current which is $d d^c$-exact.
\item For any curve $C$ in $M$, one has $K_M \cdot C \geq 0$.  In the case of negative curvature operator, $(M^{2n}, J)$ is K\"ahler, with ample canonical bundle $K_M$.
\end{enumerate}
In the case of complex surfaces $(n=2)$, case \emph{(2)} must hold.
\end{thm}

\noindent As remarked above, this theorem holds when $h$ only has negative real bisectional curvature. As a consequence, we partially confirm (\cite{yang2019real} Conjecture 1.6 (d)) when the complex dimension $n=2$. Namely, when $(M^4,J,h)$ is a compact complex surface with negative real bisectional curvature, then the canonical line bundle $K_M$ is ample.

The proof of Theorem \ref{t:mainthm} uses a special version of Hermitian curvature flow, namely the pluriclosed flow \cite{PCF}.  Condition (1) will hold if there is no pluriclosed metric on $(M^{2n}, J)$ by a result of Egidi \cite{Egidi}.  Assuming there is a pluriclosed metric (notably unrelated to the given metric $h$), we show that the solution to normalized pluriclosed flow with any initial data exists on $[0,\infty)$.  In the case of negative curvature operator we show subconvergence of the flow to a K\"ahler current.  Given this it follows from Demailly-Paun \cite{DemaillyPaun} that $(M^{2n}, J)$ is Fujiki class $\mathcal C$, then we apply a result of Chiose \cite{chiose2014} to conclude that $(M^{2n}, J)$ is in fact K\"ahler.  Knowing that $(M^{2n}, J)$ is K\"ahler, we apply the adaptation of the Wu-Yau result \cite{wuyau2016} to the case of a Hermitian metric with negative curvatures on a K\"ahler background (\cite{TosattiYang2017, DiverioTrapani2019,WuYau2016-1}, building to \cite{yang2019real}) In the case of complex surfaces, pluriclosed metrics always exist by Gauduchon's Theorem (\cite{Gauduchon1form}), and so case $(2)$ must hold.

We note that all of the rigidity results obtained by Ricci flow mentioned above rely on maximum principle arguments which show that the relevant curvature condition is preserved along the flow, and moreover improves rapidly, yielding global existence convergence to a canonical model.  Furthermore, the two recent applications of Hermitian curvature flow mentioned above (\cite{LeeNeg, Ustinovskiy}) are strong maximum principle arguments which only rely on the fact that the flows preserve certain curvature conditions.  Specifically, they are perturbative in nature, and do not require understanding the long time behavior of the flow.  Here the presence of a metric with negative curvature operator is used as a background against which to compare the evolving metric for arbitrary initial data  These estimates are key in establishing the global existence of the pluriclosed flow, which is crucial in obtaining the limiting K\"ahler current, and hence the rigidity of the underlying complex manifold.

\vskip 0.1in
\textbf{Acknowledgements:} The authors would like to thank Saikee Yeung and Jean-Pierre Demailly for helpful comments.

\section{Proofs of Theorems}
\subsection{Reduction of Pluriclosed flow}

As stated in the introduction, the main technical input into the proofs is to establish global existence and weak convergence of the pluriclosed flow on Hermitian backgrounds of nonpositive curvature operator.  To begin we recall basic aspects of this equation.  Given $(M^{2n}, J)$, a Hermitian metric $\gw$ is pluriclosed if
\begin{align*}
d d^c \gw = 0.
\end{align*}
A one-parameter family of pluriclosed metrics $\gw_t$ is a solution of normalized pluriclosed flow \cite{PCF} if
\begin{align} \label{PCF}
\dt \omega =&\ \del \del^*_{g} \gw + \delb \delb^* \gw - \rho_C(\omega) - \gw,
\end{align}
where $\rho_C(\gw)$ denotes the representative of the first Chern class with respect to the Chern connection associated to $\omega$.  Pluriclosed metrics are not locally described by potential functions, rather by potential $(1,0)$-forms.  Thus we cannot reduce (\ref{PCF}) to a scalar PDE, but it is possible to reduce to a PDE for a $(1,0)$ form, and this was carried out in general in \cite{StreetsPCFBI}.  We recall this construction here, with an important modification that we choose a one-parameter family of background tensors which are \emph{not} a priori positive definite.  To that end, given $g_t$ some solution to pluriclosed flow, choose $h$ some background Hermitian metric and consider the one-parameter family of tensors
\begin{align*}
\hat{\gw}_t = e^{-t} \gw_0 - \left(1 - e^{-t} \right) \rho_C(h).
\end{align*}
Though for large $t$ these may not be positive definite, we keep the notation $\hat{\gw}$ for simplicity, and furthermore define
\begin{align} \label{f:backgrndtorsion}
T_g = \del \gw, \qquad \hat{T} = \del \hat{\gw} = e^{-t} \del \gw_0.
\end{align}
Furthermore, let $(\gb_t, f_t)$ be the unique solution of
\begin{gather} \label{f:reducedPCF}
\begin{split}
\dt \gb =&\ \gD_{g_t} \gb - T_{g_t} \circ \delb \gb + \i \tr_{g_{\ga}} \hat{T} - \gb, \qquad \gb_0 = 0,\\
\dt f =&\ \gD_{g_t} f + \tr_{g_t} \hat{g} + \log \frac{\det g_t}{\det h} - f, \qquad f_0 = 0.
\end{split}
\end{gather}
Following the computation of (\cite{StreetsPCFBI} Proposition 3.9), the one-form $\ga_t = \gb_t - \i \del f_t$ satisfies
\begin{align*}
\dt \ga =&\ \gD_{g_t} \ga - T_{g_{\ga}} \circ \delb \ga + \i \tr_{g_t} \hat{T} - \i \del \left( \tr_{g_t} \hat{g} + \log \frac{\det g_t}{\det h} \right) - \ga
\end{align*}
Comparing against (\cite{StreetsPCFBI} Lemma 3.6), we obtain
\begin{align*}
\dt \ga_t = \delb^*_{g_{\ga}} \gw_{\ga} - \frac{\i}{2} \del \log \frac{\det g_{\ga}}{\det h} - \ga,
\end{align*}
where $\omega_{\ga_t} = \hat{\omega}_t + \delb \ga_t + \del \bga_t$.  It follows by an elementary computation using the transgression formula for the first Chern class that $\omega_{\ga_t}$ is the unique solution to normalized pluriclosed flow (\ref{PCF}) with initial condition $\omega_0$.

\subsection{Evolution equations}

\begin{lemma} \label{l:betaev} Let $(M^{2n}, J, h)$ be a Hermitian manifold, suppose $g_t$ is a solution to normalized pluriclosed flow on $M$, and suppose $(\gb_t, f_t)$ is the associated solution of (\ref{f:reducedPCF}).  Then
\begin{gather}
\begin{split}
\left(\dt - \gD \right) \brs{\gb}_{g_t}^2 \leq&\ - \brs{\N \gb}_{g_t}^2 + 2 \Re \IP{\gb, \tr_{g_t} \hat{T}} - \brs{\gb}_{g_t}^2.
\end{split}
\end{gather}
\begin{proof} This follows from the computations of (\cite{StreetsPCFBI} Proposition 4.4) and the estimate of (\cite{ASNDGKCY} Corollary 6.5).
\end{proof}
\end{lemma}

\begin{lemma} \label{l:traceev} Let $(M^{2n}, J, h)$ be a Hermitian  manifold, and suppose $g_t$ is a solution to normalized pluriclosed flow on $M$.  Then
\begin{align*}
\left(\dt - \gD \right) \tr_g h =&\ - \brs{\gU(g,h)}_{g^{-1}, g^{-1}, h}^2 - \IP{h,Q} + \Omega^h(g^{-1}, g^{-1}) + \tr_g h,
\end{align*}
where $Q_{i \bj} = T_{i k \bq} T_{\bj \bl p} g^{\bl k} g^{\bq p} \geq 0$.
\begin{proof} This follows from (\cite{StreetsPCFBI} Lemma 6.2), incorporating the normalization term.
\end{proof}
\end{lemma}

\subsection{Proofs}

\begin{proof}[Proof of Theorem \ref{t:mainthm}] If item (1) does not hold, then by (\cite{Egidi} Theorem 3.3(3)), $(M^{2n}, J)$ admits pluriclosed metrics.  The first main goal is to show that, with any initial metric, the solution to normalized pluriclosed flow exists for $[0,\infty)$.  Let $h$ denote the given metric satisfying (\ref{f:curvcond}).  Then by applying Lemma \ref{l:traceev} for this choice of $h$, and using that $Q \geq 0$, we obtain
\begin{gather}
\begin{split}
\left(\dt - \gD \right) \tr_g h \leq&\ - \gd \brs{g^{-1}}^2_h + \tr_g h\\
\leq&\ - \frac{\gd}{n} (\tr_g h)^2 + \tr_g h.
\end{split}
\end{gather}
We note that in this step the condition of nonpositive real bisectional curvature will suffice.  Applying the maximum principle, we see that $\tr_g h$ has an upper bound which is uniform in time if $\gd > 0$, and grows exponentially if $\gd = 0$.  Thus there is a uniform lower bound for the metric along the flow at all finite times, and the higher regularity follows from (\cite{StreetsPCFBI}) Theorem 1.8).  Having established the long-time existence, for any curve $C$ in $M$ we compute
\begin{align*}
\frac{d}{dt} \Area_{g_t}(C) =&\ \frac{d}{dt} \int_C \gw_t\\
=&\ - \int_C d d^*_{\gw_t} \gw + \rho_C - \Area_{g_t}(C)\\
=&\ K \cdot C - \Area_{g_t}(C)
\end{align*}
Since the area must remain positive for all finite times, it follows that $K \cdot C \geq 0$.  In case $\gd > 0$, as explained above there is a uniform lower bound for the metric.  This implies that the area of any curve must remain bounded away from zero, and thus $K \cdot C > 0$.

Next we prove that the flow converges subsequentially to a K\"ahler current.  The key estimate is exponential decay of $\beta$.  Combining Lemma \ref{l:betaev} with (\ref{f:backgrndtorsion}) and the uniform lower bound for the metric we obtain
\begin{align*}
\left(\dt - \gD \right) \brs{\gb}^2_{g_t} \leq&\ 2 \brs{\gb}_{g_t} \brs{\tr_{g_t} \hat{T}}_{g_t} - \brs{\gb}_{g_t}^2\\
\leq&\ 2 \brs{\gb}_{g_t} e^{-t} \brs{\tr_{g_t} \del \omega_0}_{g_t} - \brs{\gb}_{g_t}^2\\
\leq&\ C e^{-t} \brs{\gb}_{g_t} (\tr_{g_t} g_0)^{\frac{3}{2}} - \brs{\gb}_{g_t}^2\\
\leq&\ C e^{-t} \brs{\gb}_{g_t} - \brs{\gb}_{g_t}^2.
\end{align*}
From the maximum principle we obtain the a priori estimate
\begin{align} \label{f:betadecay}
\sup_{M \times \{t\}} \brs{\gb}_{g_t}^2 \leq&\ \left( C + \sup_{M \times \{0\}} \brs{\gb}_{g_0}^2 \right) e^{-\tfrac{t}{2}}.
\end{align}
Using this, we first show a uniform integral estimate for the metrics along the flow.  By Gauduchon's theorem \cite{Gauduchon1form} we may choose a background metric $\gw_G$ such that $\i \del \delb \gw_G^{n-1} = 0$.  Using this, integrating by parts, and applying (\ref{f:betadecay}), we estimate
\begin{align*}
\int_{M} \gw_{\ga} \wedge \gw_G^{n-1} =&\ \int_M \left(\hat{\gw} + \delb \gb + \del \bar{\gb} + \i \del\delb f \right) \wedge \gw_G^{n-1}\\
\leq&\ C + \Re \int_M \gb \wedge \delb \gw_G^{n-1}\\
\leq&\ C + \int_M \brs{\gb}_{g_t} \brs{\delb \gw_G^{n-1}}_{g_t} dV_{g_t}\\
\leq&\ C + e^{-\tfrac{t}{4}} \int_M \brs{\delb \gw_G^{n-1}}_{g_t} dV_{g_t}\\
\leq&\ C + C e^{-\tfrac{t}{4}} \int_M \sqrt{\tr_{\gw_G} \gw_{\ga}} \gw_G^n\\
\leq&\ C + C e^{-\tfrac{t}{4}} \left(\int_M \tr_{\gw_G} \gw_\ga \gw_G^n \right)^{\frac{1}{2}}\\
=&\ C + C e^{-\tfrac{t}{4}} \left( 1 + \int_M \gw_{\ga} \wedge \gw_G^{n-1} \right).
\end{align*}
Thus, for sufficiently large $t$ we can absorb the extra term back on the left to obtain a uniform estimate
\begin{align*}
\int_{M} \gw_{\ga} \wedge \gw_G^{n-1} \leq C.
\end{align*}
By the Banach-Alaoglu Theorem (cf. \cite{DemaillyBook} Chapter III Proposition 1.23), for any sequence $\{t_j\} \to \infty$ we can choose a subsequence such that $\gw_{\ga_{t_j}}$ converges weakly to a limiting current $\gw_{\infty}$.  To obtain that $\gw_{\infty}$ is closed, we estimate similarly to the above, 
\begin{align*}
\brs{\int_{M} \gw_{\ga} \wedge d \psi} =&\ \brs{\int_M d \gw_{\ga} \wedge \psi}\\
=&\ \brs{\int_M d \hat{\gw} \wedge \psi + \Re \int_M \del \delb \gb \wedge \psi}\\
\leq&\ C e^{-t} + \brs{\int_M \gb \wedge \del \delb \psi}\\
\leq&\ C e^{-t} + C e^{-\tfrac{t}{4}} \left(1 + \int_M \gw_{\ga} \wedge \gw_G^{n-1} \right)\\
\leq&\ C e^{-\tfrac{t}{4}}.
\end{align*}
Thus the limiting current is closed. It now follows from (\cite{DemaillyPaun} Theorem 0.7) that $(M^{2n}, J)$ is of Fujiki class $\mathcal C$.  However, since $M$ supports a pluriclosed metric, it follows from (\cite{chiose2014} Theorem 0.2) that $M$ is K\"ahler.  Since $(M, J)$ is now K\"ahler, it follows from (\cite{yang2019real} Theorem 1.8) that the canonical bundle is ample.

In the case $n=2$, by Gauduchon's Theorem \cite{Gauduchon1form} we know that there exist pluriclosed metrics on $(M^4, J)$, and so case (2) must hold.

\end{proof}

\bibliographystyle{acm}

\begin{thebibliography}{10}

\bibitem{ASNDGKCY}
{\sc {Apostolov}, V., and {Streets}, J.}
\newblock {The nondegenerate generalized K\"ahler Calabi-Yau problem}.
\newblock {\em ArXiv e-prints\/} (Mar. 2017).

\bibitem{BohmWilking}
{\sc B\"{o}hm, C., and Wilking, B.}
\newblock Manifolds with positive curvature operators are space forms.
\newblock {\em Ann. of Math. (2) 167}, 3 (2008), 1079--1097.

\bibitem{BrendleSchoen}
{\sc Brendle, S., and Schoen, R.}
\newblock Manifolds with {$1/4$}-pinched curvature are space forms.
\newblock {\em J. Amer. Math. Soc. 22}, 1 (2009), 287--307.

\bibitem{ChenSunTian}
{\sc Chen, X., Sun, S., and Tian, G.}
\newblock A note on {K}\"{a}hler-{R}icci soliton.
\newblock {\em Int. Math. Res. Not. IMRN}, 17 (2009), 3328--3336.

\bibitem{chiose2014}
{\sc Chiose, I.}
\newblock Obstructions to the existence of k{\"a}hler structures on compact
  complex manifolds.
\newblock {\em Proceedings of the American Mathematical Society 142}, 10
  (2014), 3561--3568.

\bibitem{DemaillyBook}
{\sc Demailly, J.-P.}
\newblock {\em Complex analytic and differential geometry}.

\bibitem{DemaillyPaun}
{\sc Demailly, J.-P., and Paun, M.}
\newblock Numerical characterization of the {K}\"ahler cone of a compact
  {K}\"ahler manifold.
\newblock {\em Ann. of Math. (2) 159}, 3 (2004), 1247--1274.

\bibitem{DiverioTrapani2019}
{\sc Diverio, S., Trapani, S., et~al.}
\newblock Quasi-negative holomorphic sectional curvature and positivity of the
  canonical bundle.
\newblock {\em Journal of Differential Geometry 111}, 2 (2019), 303--314.

\bibitem{Egidi}
{\sc Egidi, N.}
\newblock Special metrics on compact complex manifolds.
\newblock {\em Differential Geom. Appl. 14}, 3 (2001), 217--234.

\bibitem{Gauduchon1form}
{\sc Gauduchon, P.}
\newblock La {$1$}-forme de torsion d'une vari\'et\'e hermitienne compacte.
\newblock {\em Math. Ann. 267}, 4 (1984), 495--518.

\bibitem{GuFrankel}
{\sc Gu, H.-L.}
\newblock A new proof of {M}ok's generalized {F}rankel conjecture theorem.
\newblock {\em Proc. Amer. Math. Soc. 137}, 3 (2009), 1063--1068.

\bibitem{guenancia2018}
{\sc Guenancia, H.}
\newblock Quasi-projective manifolds with negative holomorphic sectional
  curvature.
\newblock {\em arXiv preprint arXiv::1808.01854\/} (2018).

\bibitem{Hamilton3folds}
{\sc Hamilton, R.~S.}
\newblock Three-manifolds with positive {R}icci curvature.
\newblock {\em J. Differential Geom. 17}, 2 (1982), 255--306.

\bibitem{Hamilton4folds}
{\sc Hamilton, R.~S.}
\newblock Four-manifolds with positive curvature operator.
\newblock {\em J. Differential Geom. 24}, 2 (1986), 153--179.

\bibitem{HeierLuWong2016}
{\sc Heier, G., Lu, S.~S., Wong, B., et~al.}
\newblock K{\"a}hler manifolds of semi-negative holomorphic sectional
  curvature.
\newblock {\em Journal of Differential Geometry 104}, 3 (2016), 419--441.

\bibitem{HeierLuWongZheng2018}
{\sc Heier, G., Lu, S.~S., Wong, B., and Zheng, F.}
\newblock Reduction of manifolds with semi-negative holomorphic sectional
  curvature.
\newblock {\em Mathematische Annalen 372}, 3-4 (2018), 951--962.

\bibitem{LeeNeg}
{\sc Lee, M.-C.}
\newblock Compact hermitian manifolds with quasi-negative curvature.
\newblock {\em arXiv preprint arXiv:1810.07325\/} (2018).

\bibitem{Margerin}
{\sc Margerin, C.}
\newblock A sharp characterization of the smooth {$4$}-sphere in curvature
  terms.
\newblock {\em Comm. Anal. Geom. 6}, 1 (1998), 21--65.

\bibitem{MokFrankel}
{\sc Mok, N.}
\newblock The uniformization theorem for compact {K}\"{a}hler manifolds of
  nonnegative holomorphic bisectional curvature.
\newblock {\em J. Differential Geom. 27}, 2 (1988), 179--214.

\bibitem{NiWilking}
{\sc Ni, L., and Wilking, B.}
\newblock Manifolds with {$1/4$}-pinched flag curvature.
\newblock {\em Geom. Funct. Anal. 20}, 2 (2010), 571--591.

\bibitem{StreetsPCFBI}
{\sc Streets, J.}
\newblock Pluriclosed flow, {B}orn-{I}nfeld geometry, and rigidity results for
  generalized {K}\"ahler manifolds.
\newblock {\em Comm. Partial Differential Equations 41}, 2 (2016), 318--374.

\bibitem{PCF}
{\sc Streets, J., and Tian, G.}
\newblock A parabolic flow of pluriclosed metrics.
\newblock {\em Int. Math. Res. Not. IMRN}, 16 (2010), 3101--3133.

\bibitem{HCF}
{\sc Streets, J., and Tian, G.}
\newblock Hermitian curvature flow.
\newblock {\em J. Eur. Math. Soc. (JEMS) 13}, 3 (2011), 601--634.

\bibitem{TosattiYang2017}
{\sc Tosatti, V., Yang, X., et~al.}
\newblock An extension of a theorem of Wu--Yau.
\newblock {\em Journal of Differential Geometry 107}, 3 (2017), 573--579.

\bibitem{Ustinovskiy}
{\sc {Ustinovskiy}, Y.}
\newblock {The Hermitian curvature flow on manifolds with non-negative
  Griffiths curvature}.
\newblock {\em ArXiv e-prints\/} (Apr. 2016).

\bibitem{wuyau2016}
{\sc Wu, D., and Yau, S.-T.}
\newblock Negative holomorphic curvature and positive canonical bundle.
\newblock {\em Inventiones mathematicae 204}, 2 (2016), 595--604.

\bibitem{WuYau2016-1}
{\sc Wu, D., and Yau, S.-T.}
\newblock A remark on our paper" negative holomorphic curvature and positive
  canonical bundle".
\newblock {\em Comm. Anal. Geom. 24 (2016), no. 4, 901–912.}




\bibitem{yang2019real}
{\sc Yang, X., and Zheng, F.}
\newblock On real bisectional curvature for hermitian manifolds.
\newblock {\em Transactions of the American Mathematical Society 371}, 4
  (2019), 2703--2718.

\end{thebibliography}

\end{document}